\documentclass[12pt,reqno]{amsart}
\usepackage{latexsym}
\usepackage{amssymb}
\usepackage{mathrsfs}
\usepackage{amsmath}
\usepackage{comment}
\usepackage{fancybox,color}
\usepackage{enumerate}
\usepackage[latin1]{inputenc}
\usepackage{soul} 
\usepackage[colorlinks=true, linkcolor=blue, citecolor=blue]{hyperref}
\usepackage{graphicx}
\usepackage{autonum}
\allowdisplaybreaks[4]
{\catcode`p =12 \catcode`t =12 \gdef\eeaa#1pt{#1}}      
\def\accentadjtext#1{\setbox0\hbox{$#1$}\kern   
	\expandafter\eeaa\the\fontdimen1\textfont1 \ht0 }
\def\accentadjscript#1{\setbox0\hbox{$#1$}\kern 
	\expandafter\eeaa\the\fontdimen1\scriptfont1 \ht0 }
\def\accentadjscriptscript#1{\setbox0\hbox{$#1$}\kern   
	\expandafter\eeaa\the\fontdimen1\scriptscriptfont1 \ht0 }
\def\accentadjtextback#1{\setbox0\hbox{$#1$}\kern       
	-\expandafter\eeaa\the\fontdimen1\textfont1 \ht0 }
\def\accentadjscriptback#1{\setbox0\hbox{$#1$}\kern     
	-\expandafter\eeaa\the\fontdimen1\scriptfont1 \ht0 }
\def\accentadjscriptscriptback#1{\setbox0\hbox{$#1$}\kern 
	-\expandafter\eeaa\the\fontdimen1\scriptscriptfont1 \ht0 }

\newcommand{\ch}[1]{{\mbox{\raise 1pt\hbox{\large$\chi$}}}_{\lower 1pt\hbox{$\scriptstyle #1$}}}

\usepackage[left=2.1cm,top=2.3cm,right=2.1cm]{geometry}

\def\1{\raisebox{2pt}{\rm{$\chi$}}}

\newtheorem{theorem}{Theorem}[section]

\newtheorem{corollary}[theorem]{Corollary}
\newtheorem{lemma}[theorem]{Lemma}

\theoremstyle{definition}
\newtheorem{definition}[theorem]{Definition}
\newtheorem{remark}[theorem]{Remark}

\newcommand{\R}{{\mathbb R}}

\newcommand{\N}{{\mathbb N}}
\newcommand{\Z}{{\mathbb Z}}

\newcommand\diam{\operatorname{diam}}

\def\1{\raisebox{2pt}{\rm{$\chi$}}}

\def\vint_#1{\mathchoice%
	{\mathop{\kern 0.2em\vrule width 0.6em height 0.69678ex depth -0.58065ex
			\kern -0.8em \intop}\nolimits_{\kern -0.4em#1}}%
	{\mathop{\kern 0.1em\vrule width 0.5em height 0.69678ex depth -0.60387ex
			\kern -0.6em \intop}\nolimits_{#1}}%
	{\mathop{\kern 0.1em\vrule width 0.5em height 0.69678ex
			depth -0.60387ex
			\kern -0.6em \intop}\nolimits_{#1}}%
	{\mathop{\kern 0.1em\vrule width 0.5em height 0.69678ex depth -0.60387ex
			\kern -0.6em \intop}\nolimits_{#1}}}
\def\vintslides_#1{\mathchoice%
	{\mathop{\kern 0.1em\vrule width 0.5em height 0.697ex depth -0.581ex
			\kern -0.6em \intop}\nolimits_{\kern -0.4em#1}}%
	{\mathop{\kern 0.1em\vrule width 0.3em height 0.697ex depth -0.604ex
			\kern -0.4em \intop}\nolimits_{#1}}%
	{\mathop{\kern 0.1em\vrule width 0.3em height 0.697ex depth -0.604ex
			\kern -0.4em \intop}\nolimits_{#1}}%
	{\mathop{\kern 0.1em\vrule width 0.3em height 0.697ex depth -0.604ex
			\kern -0.4em \intop}\nolimits_{#1}}}

\newcommand{\intav}{\vint}

\newcommand{\dist}{\operatorname{dist}}

\title[Characterizing Capacity]{Sharp Capacity Estimates and Volume Growth for Fractional Sobolev Spaces}

\author[A. Agarwal]{Angha Agarwal}   %
\address[A.I.]{Department of Mathematics and Statistics, P.O. Box 35, FI-40014 University of Jyvaskyla, Finland}
\email{angha.a.agarwal@jyu.fi}



\keywords{fractional Sobolev space, capacity, volume growth}

\subjclass[2020]{46E35}

\begin{document}

\begin{abstract}
   We obtain sharp two-sided estimates for the relative capacity of fat and thin
annuli in the fractional Sobolev spaces $W^s_{p,q}$, where $s\in(0,1)$, $p\in(0,\infty)$, $q\in(0,\infty]$,
 in terms of an explicit volume growth quantity. The underlying
space is a general metric measure space equipped with a doubling and reverse
doubling measure. In particular, whenever the capacity is finite, the
estimates are independent of the parameter $q$, and in the Ahlfors regular case with $p=q$ they recover Costea's logarithmic law at $sp=Q$. For $p<q$, we also introduce a sharp condition on the measure that is necessary for the existence of admissible functions in the space $W^s_{p,q}$. For thin annuli $(B_r,B_{r+\varepsilon r})$  under an
$\eta$-annular decay property, we uncover a purely nonlocal phenomenon: for $sp <\eta$, the capacity remains bounded as $\varepsilon\to 0$, because the limiting competitor $\1_{B_r}$
already has finite energy, while if $sp\geq\eta$ it  blows up as $ \varepsilon\to0$.
	\end{abstract}
    \maketitle
\section{Introduction}
In this paper, we study capacities associated with the fractional Sobolev spaces $W^s_{p,q}$ with mixed $(L^p,L^q)$ integrability on metric measure spaces. For $s\in(0,1)$, $p\in(0,\infty)$ and $q\in(0,\infty]$ we consider the (quasi) seminorm
\[
	[u]_{W^s_{p,q}(X)} = \left(\int_{X}\left(\int_{X}\frac{|u(x)-u(y)|^q}{d(x,y)^{s q}\mu(B(x,d(x,y)))}\,d\mu(y)\right)^{\frac{p}{q}}\,d\mu(x)\right)^\frac{1}{p},
	\]
with the usual modification when $q=\infty$. When $p=q$ this is the Gagliardo-type energy, which for $p\ge 1$ coincides with
the classical Besov seminorm $B^s_{p,p}$ (see \cite{RuSi,Triebel95}). For $p\neq q$, the
the resulting seminorm is of
Triebel--Lizorkin type. In a forthcoming
work \cite{AljA}, we show that spaces 
$W^s_{p,q}$ can be identified with appropriate Triebel--Lizorkin spaces
(see also \cite{KoYaZh}). Such mixed-norm spaces have attracted considerable
attention in recent years because of their role in nonlocal partial differential
equations, harmonic analysis, geometric function theory and nonlinear potential
theory (\cite{PrEe, LiKaVa, KoYaZh, DiElVa, Warma15, AdDa, AnPeKa}). Capacity
problems for functions in local and nonlocal spaces have likewise been studied
extensively (\cite{Costea,BB,HoPe,HeJuKiTeMaOl,BBJ,Adams,KaMo}).

Capacity measures the
size of a set relative to the function space, and in the fractional setting
it is not always finite; for certain choices of the parameters it may
be infinite. Our aim is therefore to understand the behavior of the
capacity of annuli when it is finite. We obtain sharp two-sided estimates,
and show that the capacity is comparable to an explicit volume growth
quantity and to quantities depending only on the energies of indicator
functions of dyadic balls.

Related problems have been studied in the case   $p=q\ge1$ in \cite{BB,Costea}.
In \cite{BB} the
authors obtain sharp capacity bounds for fat annuli by interpreting Besov spaces as trace spaces of
first-order Sobolev spaces. In \cite{Costea}, the author studied these
capacities in the borderline case $sp=Q$, where $Q$ is the Ahlfors
regularity dimension of the space. His approach is based on Fubini's
theorem. Neither approach extends directly to the
general $W^s_{p,q}$ setting when $p\neq q$;  see  \cite{BBN}. In this article, we cover the case $p\ne q$ and establish sharp
capacity estimates for both fat and thin
annuli in a more general setting, by direct and self-contained arguments. 

A difficulty when $p\neq q$ is the lack of symmetry in the norm, which prevents us from interchanging the order of integration as in the case $p=q$. This makes the long-range interactions more delicate to control and requires additional assumptions on the underlying space.

\medskip

\emph{Throughout the paper,} we assume that $(X,d,\mu)$ is a metric measure
space with $\mu$ doubling and reverse doubling (see Section~\ref{prelim}). We
fix a point $O\in X$ and write $B_r=B(O,r)$.
We introduce a \emph{dyadic condition} \eqref{Rs,p,q finiteness} on the underlying measure as
\begin{equation}
    D_{s,p,q}(\mu)=  \sum_{j=1}^\infty \frac{\mu(B(O,2^j))^{1-\frac{p}{q}}}{2^{jsp}}<\infty.\tag{D1}
\end{equation}
 If this sum is infinite, we show that there are no admissible
functions in $W^s_{p,q}(X)$ (see Lemma~\ref{necessary condition}). To obtain sharp two-sided capacity estimates, we assume a stronger
\emph{dyadic tail condition} \eqref{eq2_6.9}, which controls the tail of the sum
$D_{s,p,q}$ by its first term, up to a constant. Since $\mu$ is doubling and
reverse doubling, it has an upper growth exponent, say $Q$, and a lower growth
exponent, say $\sigma$, which need not coincide. The dyadic tail condition is
natural in this setting because it reflects the interaction between these two
growth rates, rather than a single dimension. When 
$s>\frac {Q}{p}-\frac{Q}{q}$, the dyadic tail condition automatically holds. Our estimates therefore apply to a broad class of
metric measure spaces, including irregular spaces in which the upper and lower
growth rates differ. We discuss both conditions \eqref{Rs,p,q finiteness} and \eqref{eq2_6.9} in more detail in Section~\ref{UB estimate}. 

We emphasize that no completeness, Poincar\'e inequality, or boundedness
assumption is imposed on the space. Under our assumptions, we
establish sharp bounds for the $\dot W^s_{p,q}(X)$-relative capacity, assuming
the dyadic tail condition \eqref{eq2_6.9} when $p<q$. For a precise definition of capacity, we refer
to Definition~\ref{capacity}. We now state our main results.

\begin{theorem}\label{sharp capacity estimate annuli}
Let $(X,d,\mu)$ be a metric measure
space with $\mu$ doubling and reverse doubling, and let $0<s<1$ and $0<q\le p<\infty$. 
 Then, for $0<2r\le R<\frac{1}{4}\diam(X)$ we have
   \begin{equation}\label{2.2_5.9}
\mathrm{cap}_{s,p,q}(B_r,B_R)
\simeq
\left(\int_{r}^{R}\left(\frac{\hat{r}^{sp}}{\mu(B_{\hat{r}})}\right)^{\frac{1}{p-1}}\frac{d\hat{r}}{\hat{r}}\right)^{1-p },\text{   }p>1,
\end{equation}
and 
\begin{equation}\label{2.3_5.9}
\mathrm{cap}_{s,p,q}(B_r,B_R)\simeq 
\inf_{r\le \hat{r}<2\hat{r}\le R}\left(\frac{\hat{r}^{sp}}{\mu(B_{\hat{r}})}\right)^{-1 },\text{   }p\le 1,
\end{equation}
where the comparison constant depends on $s$, $p$, and $C(c_D,c_{RD},C_{DT})^{1/q}$. 

Moreover, the above estimates also hold for $0<p<q\le\infty$, provided that the measure $\mu$ satisfies the dyadic tail condition \eqref{eq2_6.9}.
\end{theorem}

Theorem~\ref{sharp capacity estimate annuli} shows that the $(s,p,q)$-capacity is comparable to the $B^s_{p,p}$-capacity of fat annuli, and the exponent $q$ does not really play a role in the estimates. This is somewhat similar to the situation in the classical theory
of Adams--Hedberg (see \cite[Section~4.4]{AdDa}). 

It is also somewhat
surprising that the upper bound is harder to obtain than the lower bound;
an upper bound only requires a single competitor, but here that competitor
must have finite energy, and proving this requires controlled 
long-range interactions, which is where the dyadic conditions enter.



\medskip
We postpone the detailed capacity estimates for thin annuli to
Section~\ref{thin annuli section}, since they require more technical
arguments and rely on a thin annular decay property. The following corollary
is a simplified statement, covering both fat and thin annuli, of
Corollary~\ref{ind_fat annuli} and Corollary~\ref{ind_thin annuli}. It
follows from Theorem~\ref{sharp capacity estimate annuli} and
Theorem~\ref{thin annuli_main}, and it gives the reader a preview of the form
of the thin annuli estimates.

\begin{corollary}\label{coro_14.9}
Let $X=\R^n$ be equipped with the Euclidean metric and the Lebesgue measure $\mu$. Let $p\in(0,\infty)$, $q\in(0,\infty]$, $\max\{0,\frac{n}{p}-\frac{n}{q}\}<s<1$  and  $sp<1$. Then for $0<2r\leq R<\infty$ we have
    \begin{equation}
              \mathrm{cap}_{s,p,q}(B_r,B_R)\simeq 
              \begin{cases}
              \displaystyle
     \left(\int_{r}^{R}\,[\1_{B_{\hat{r}}}]_{{W}^s_{p,q}(\R^n)}^{\frac{- p}{p-1}}\frac{d\hat{r}}{\hat{r}}\right)^{1-p }, & p>1,\\ \\
\displaystyle
   \underset{r\leq \hat r<2\hat r\leq R}{\inf}\,[\1_{B_{\hat{r}}}]_{{W}^s_{p,q}(\R^n)}^p\qquad, & p\le 1,
\end{cases}
            \end{equation}
            and for $0<\frac{R}{2}\le r<R<\infty$ we have
            \begin{equation}
              \mathrm{cap}_{s,p,q}(B_r,B_{R})\simeq [\1_{B_{r}}]_{{W}^s_{p,q}(\R^n)}^p.
            \end{equation}
            The comparison constants above depend on $s$, $p$, and $C(c_D,c_{RD},C_{DT})^{1/q}$.
\end{corollary}
Corollary~\ref{coro_14.9} reduces the study of capacity to that of the
$\dot{W}^s_{p,q}$-energies of indicator functions of balls, whose behavior across
scales is governed by the scaling properties of the norm. This allows us to estimate the capacity without first determining the growth exponents of the measure, which may be difficult to compute for a given doubling and reverse doubling measure.

\medskip

We now state the following consequences of the estimate obtained in Theorem~\ref{sharp capacity estimate annuli} that may be of interest.

\begin{corollary}[Properties of the Capacity]\label{prop0_13.9}
In the setting of Theorem~\ref{sharp capacity estimate annuli}, we give the following properties:
\begin{itemize}
    \item[(i)] If
     \begin{equation}
\int_{0}^{R}\left(\frac{\hat{r}^{sp}}{\mu(B_{\hat{r}})}\right)^{\frac{1}{p-1}}\frac{d\hat{r}}{\hat{r}}=\infty, \quad p>1, \text{  and  }\limsup_{r\to0}\frac{r^{sp}}{\mu(B_r)}=\infty,\quad p\le 1,
    \end{equation}
    then  $\mathrm{cap}_{s,p,q}(\{O\},B_R)=0$.\vspace*{0.3cm}
   \item[(ii)] Let $X$ be an unbounded metric space. If
    \begin{equation}
       \int_{1}^{\infty}\left(\frac{\hat{r}^{sp}}{\mu(B_{\hat{r}})}\right)^{\frac{1}{p-1}}\frac{d\hat{r}}{\hat{r}}=\infty, \quad p>1, \text{ and }\limsup_{r\to\infty}\frac{r^{sp}}{\mu(B_r)}=\infty,\quad p\le 1,
    \end{equation}
    then  $\mathrm{cap}_{s,p,q}(C,X)=0$, whenever $C$ is a compact subset of $X$.\vspace*{0.3cm}
   \item[(iii)]  If $sp=s'p'$ and $0<q,q'\le\infty$, then $\mathrm{cap}_{s,p,q}(B_r,B_R)\simeq\mathrm{cap}_{s',p',q'}(B_r,B_R)\simeq\mathrm{cap}_{s,p,p}(B_r,B_R)$.
\end{itemize}
\end{corollary}
\begin{remark}
    We note that the claim in Corollary~\ref{prop0_13.9} (iii) also holds for thin
annuli, that is, when $0<R/2\le r<R$; further results on thin annuli are
given in Section~\ref{thin annuli section}.
\end{remark}




We use the following notation throughout.  We write $a\lesssim b$ (equivalently, $b\gtrsim a$) if there exists a constant $C\in(0,\infty)$, independent of the relevant parameters, such that $a\leq C b$. We write $a\approx b$ if $a\lesssim b$ and $b\lesssim a$. The characteristic function of a set $E\subset X$ is denoted by
$ \1_E$, that is, $\1_E(x)=1$ for $x\in E$ and $ \1_E(x)=0$
for $x\in X\setminus E$. We denote by $L^0(X)$ the space of those $\mu$-measurable
functions $u\colon X\to[-\infty,\infty]$ that are finite $\mu$-almost everywhere.

 \medskip

The paper is organized as follows. Section~\ref{prelim} contains the
definitions and the known results needed in the rest of the paper. In
Section~\ref{sharp capacity}, we establish sharp capacity estimates for fat
annuli. In Section~\ref{thin annuli section}, we estimate the capacities of
thin annuli under the thin annular decay property, and we characterize the
capacities of fat and thin annuli through indicator functions.

 \section{Preliminiaries}\label{prelim}
We introduce our standing assumptions on the metric space and the measure. Let $(X,d,\mu)$ be a metric measure space. Here $\mu$ is a positive  complete Borel measure on $X$ such that \[0<\mu(B(x,r))<\infty\] for all $x\in X$ and $r>0$, where  $B(x,r)=\{y\in X\,:\, d(y,x)<r\}$  is the open ball of radius $r>0$ and center $x\in X$.
The measure $\mu$ is  assumed to be doubling, that is, there exists a constant $c_{D}\geq1$ such that, whenever $x\in X$ and $r>0$, we have
\begin{equation}\label{doubling}
    \mu(B(x,2r))\leq c_{D}\,\mu(B(x,r)).
\end{equation}
We also assume that 
the measure $\mu$ is reverse doubling, that is, there exists a constant $0<c_{RD}<1$ such that
\begin{equation}\label{reverse doubling}
\mu(B(x,r))\le c_{RD}\,   \mu(B(x, 2r))  
\end{equation}
for all $x\in X$ and all $0<2 r<\mathrm{diam}(X)$.\\

We also give the following iterated versions
of the doubling and reverse doubling conditions.
Iteration of the doubling condition \eqref{doubling} gives an exponent $Q>0$ and a constant $c_{Q}>0$, both only depending on $c_{D}$, such that
\begin{equation}\label{iteration doubling}
   \frac{\mu(B(y,r))}{\mu(B(x,r_0))}\geq c_{Q}\left(\frac{r}{r_0}\right)^{Q}
\end{equation}
whenever $x\in X$, $y\in B(x,r_0)$ and $0<r<r_0$. 
On the other hand, the doubling property and iteration of the reverse doubling condition \eqref{reverse doubling} give an exponent $\sigma>0$ and a constant $c_{\sigma}>0$, both only depending on $c_{RD}$, such that the  quantitative reverse doubling condition
\begin{equation}\label{iteration reverse}
   \frac{\mu(B(x,r))}{\mu(B(x,r_0))}\leq c_{\sigma}\left(\frac{r}{r_0}\right)^{\sigma}
\end{equation}
holds for all $x\in X$ and $0<r<r_0<\mathrm{diam}(X)$.  We refer
to the proof of \cite[Corollary 3.8]{BB}. Given $x\in X$ and $r>0$, we write $A_r(x)=B(x,2 r)\setminus B(x,r)$. Then it also follows from \eqref{reverse doubling} that $\mu(A_r(x))\approx\mu(B(x,2r))$ for $0<2r<\diam(X)$  with constants depending only on $c_{RD}$.

\subsection*{Fractional Sobolev Spaces}
Let $s\in(0,1)$, $p\in(0,\infty)$ and $q\in(0,\infty]$. The homogeneous fractional Sobolev space $\dot{W}^s_{p,q}(X)$ consists of all $u\in L^0(X)$ for which  
	\[
	[u]_{W^s_{p,q}(X)} = \left(\int_{X}\left(\int_{X}\frac{|u(x)-u(y)|^q}{d(x,y)^{s q}\mu(B(x,d(x,y)))}\,d\mu(y)\right)^{\frac{p}{q}}\,d\mu(x)\right)^\frac{1}{p}<\infty\,,
	\]
    with the usual modification when $q=\infty$. 
The Fractional Sobolev space  ${W}^s_{p,q}(X)$ is $\dot{W}^s_{p,q}(X)\cap L^p(X)$ equipped with the quasinorm 
\[
     \|u\|_{{W}^s_{p,q}(X)}=\|u\|_{L^p(X)}+[u]_{W^s_{p,q}(X)}\,.
\]

\begin{definition}\label{capacity}
Let $\Omega\subset X$ be a bounded open set such that $\mu(X\setminus \Omega)>0$. Then the $\dot{W}_{p,q}^{s}(X)$ \emph{relative capacity of a bounded  set $E\subset \Omega$} is 
	\begin{equation}
    \begin{split}
        \mathrm{cap}_{s,p,q}&(E,\Omega)=\inf_u\,\,\int_{X}\left(\int_{X}\frac{|u(x)-u(y)|^q}{d(x,y)^{sq}\mu(B(x,d(x,y)))}\,d\mu(y)\right)^{\frac{p}{q}}\,d\mu(x),
    \end{split}
	\end{equation}
where infimum is taken over all measurable functions $u$ such that $u=1$ on a neighbourhood of $E$, $0\leq u\leq 1$ for all $x\in \Omega$ and $u=0$ in $\Omega^c$. If there are no such admissible functions that satisfy the above conditions, then we set $\mathrm{cap}_{s,p,q}(E,\Omega)=\infty$.
\end{definition}

\subsection*{Medians}\label{median definition}
Let $u\in L^0(X)$ and $E\subset X$ be a set of finite positive measure. Define the median of $u$ over $E$ as
$$
m_u(E)= \inf\left\{ M\in\R \ | \ \mu\left(\{x\in E\ |\ u(x)>M\}\right) < \frac{\mu(E)}{2}\right\}.
$$


\begin{lemma}[Lemma~2.7 \cite{HeKoTu,AnPeKa}]\label{median prps}
 Let $E,F$ be measurable subsets of $X$ of finite positive measure and $u\in L^0(X)$. Then the median has the following properties:
  \begin{itemize}
      \item[(i)] If $u\le v$ almost everywhere, then $m_u(E)\le m_v(E)$.
      \item[(ii)] For every $p>0$,
      \begin{equation}
          m_{|u|}(E)\leq\left(2\intav_E|u|^p\, d\mu(x)\right)^\frac{1}{p}.
      \end{equation}
      \item[(iii)] For every $p,q>0$,
      \begin{equation}
          |m_u(E)-m_u(F)|^p\leq\,2^{1+\frac{p}{q}} \intav_E\left(\intav_F\frac{|u(x)-u(y)|^q}{d(x,y)^{sq}\mu(B(x,d(x,y)))}\,d\mu(y)\right)^\frac{p}{q}\,d\mu(x).
      \end{equation}
  \end{itemize}
\end{lemma}

 \subsection*{An inequality for $p\leq 1$}If $a_i\geq0$ for all $i$ and $0<p\leq1$, then
		\begin{equation} \label{ineq-p<1}
			\sum_{i\in\Z}a_i\leq \left(\sum_{i\in\Z}a_i^{p}\right)^{\frac{1}{p}}.
		\end{equation}
We state the following lemmas that are useful for our purpose.
\begin{lemma}\cite{BB}\label{Lemma 8.1}
Let $0<\Theta_1<\Theta_2<\infty$ and $\lambda,r>0$. If $\mu$ is doubling, then
\[
\int_{\Theta_1 r}^{\Theta_2 r}
\left(
\frac{\hat{r}^{s p}}{\mu(B_{\hat{r}})}
\right)^{\lambda}
\frac{d\hat{r}}{\hat{r}}
\simeq
\left(
\frac{r^{s p}}{\mu(B_r)}
\right)^{\lambda},
\]
with comparison constant depending only on
$\Theta_1$, $\Theta_2$, $s$, $p$, $\lambda$ and $c_D$.
\end{lemma}

\begin{lemma}\cite{BB}\label{Lip func}
   Let $0<s<1$ and $0<q<\infty$. Let $0\leq u\leq1$ be an $L$-Lipschitz function and $\mu$ be a doubling measure on $X$. Fix $x\in X$. Then we have 
  \begin{equation}
     \int_X\frac{|u(x)-u(y)|^q}{d(x,y)^{sq}\mu(B(x,d(x,y)))} d\mu(y)\lesssim L^{sq},
  \end{equation}
  where the comparison constant depends on $s$ and $q$.
\end{lemma}

\begin{lemma}\cite[Section~1.8]{KuPe03}\label{discrete weighted Hardy}
    Let $1<p<\infty$ and let $p'=\frac{p}{p-1}$ be its conjugate exponent. Consider positive  weight sequences  $\{u_n\}_{n\in\N}$, $\{v_n\}_{n\in\N}$ such that
    \begin{equation}\label{eq0_3.9}
       A:= \sup_{k>0}\left(\sum_{n=k}^{\infty} u_n\right)^\frac{1}{p}\left(\sum_{n=0}^{k} v_n^{1-p'}\right)^{1/p'}<\infty.
    \end{equation}
    Then, for any non-negative sequence $\{a_n\}_{n\in\N}$, the following inequality holds:
    \begin{equation}
       \left(\sum_{n=1}^{\infty}u_n\left(\sum_{k=1}^{n}a_k\right)^p\right)^{1/p}\leq C_p A\left(\sum_{n=1}^{\infty}v_na_n^p\right)^{1/p},
    \end{equation}
    where $C_p>0$ is a constant depending only on $p$.
\end{lemma}

\section{Sharp capacity estimates for fat annuli}\label{sharp capacity}

In this section, we give the sharp two-sided bounds for the capacity of fat annuli. We first establish the lower bound and then obtain the corresponding upper bound. Combining these estimates yields the sharp characterization in Theorem~\ref{sharp capacity estimate annuli}.

\subsection{Lower bound for capacity}\label{LB esti}
We establish a lower bound for the fractional Sobolev capacity. In contrast with the local $p$-capacity, where lower bounds are typically
obtained under the assumption that $X$ supports a $p$-Poincar\'e inequality,
no such assumption is needed here. This is because the fractional norm depends on estimating pointwise differences of the function. Consequently, lower bounds for the fractional capacity can be obtained directly from the global doubling and reverse-doubling properties of the measure.

\begin{lemma}\label{UB capacity_ball}
Let $s\in(0,1)$, $p\in(0,\infty)$ and $q\in(0,\infty]$. Then for $0<2r\leq R<\frac{1}{4}\diam(X)$  we have
\begin{equation}\label{2.2.2}
\mathrm{cap}_{s,p,q}(B_r,B_R)
\gtrsim
\left(\int_{r}^{R}\left(\frac{\hat{r}^{sp}}{\mu(B_{\hat{r}})}\right)^{\frac{1}{p-1}}\frac{d\hat{r}}{\hat{r}}\right)^{1-p},\text{   }p>1
\end{equation}
and 
\begin{equation}\label{2.3.2}
\mathrm{cap}_{s,p,q}(B_r,B_R)
\gtrsim
\inf_{r\le \hat{r}<2\hat{r}\le R}\left(\frac{\hat{r}^{sp}}{\mu(B_{\hat{r}})}\right)^{-1 },\text{   }p\le1,
\end{equation}
where the comparison constant depends on $s$, $p$, and $C(c_{RD},c_D)^{1/q}$. 
\end{lemma}
\begin{proof}
 Let $i_0\ge 1$ be a smallest integer such that $2^{i_0}r< R \le 2^{i_0+1}r$. For $i\in\{0,1,\dots,i_0+2\}$, set $B_i=B(O,2^ir)$, $A_0=B_0$ and
$A_i=B_i\setminus B_{i-1}$ for $i\ge1$.
If $\mathrm{cap}_{s,p,q}(B_r,B_{i_0+1})=\infty$, then \eqref{2.2.2} and \eqref{2.3.2}
hold. Hence, it suffices to consider the case $\mathrm{cap}_{s,p,q}(B_r,B_{i_0+1})<\infty$. 

Let $u$ be an admissible function for $\mathrm{cap}_{s,p,q}(B_r,B_{i_0+1})$. Then, $u=1$ on $B_r$ and $u=0$ in $B_{i_0+1}^c$. 
By Lemma~\ref{median prps}, we have
\begin{align}\label{eq1_24.6}
1\le|m_{{u}}(A_0)|\le& |m_{{u}}(A_0)-m_{{u}}(A_{i_0+1})|+|m_{{u}}(A_{i_0+1})|\\
\le  &  \sum_{i=0}^{i_0}|m_{{u}}(A_{i})-m_{{u}}(A_{i+1})|+|m_{{u}}(A_{i_0+1})|\\
\lesssim  & \sum_{i=0}^{i_0}\left(\intav_{A_i}\left(\intav_{A_{i+1}}|u(x)-u(y)|^q\,d\mu(y)\right)^\frac{p}{q}\,d\mu(x)\right)^\frac{1}{p}+|m_{{u}}(A_{i_0+1})|\\
 \lesssim &  \sum_{i=0}^{i_0}\frac{(2^{i}r)^s}{\mu(B_i)^\frac{1}{p}}\left(\int_{A_i}\left(\int_{A_{i+1}}\frac{|u(x)-u(y)|^q}{d(x,y)^{sq}\mu(B(x,d(x,y)))}\,d\mu(y)\right)^\frac{p}{q}\,d\mu(x)\right)^\frac{1}{p}+|m_{{u}}(A_{i_0+1})|.
\end{align}
Now we estimate $|m_{{u}}(A_{i_0+1})|$ using the properties of the median in Lemma~\ref{median prps}. Let us first fix $x\in B(O,2^{i_0+1}r)$ and $y\in A_{i_0+2}$. Then
\begin{equation}\begin{split}
    |u(x)|&=|u(x)-u(y)|\\
    &\le\left(\intav_{A_{i_0+2}} |u(x)-u(y)|^q\,d\mu(y)\right)^\frac{1}{q}\\
     &\lesssim (2^{i_0}r)^s\left(\int_{A_{i_0+2}}\frac{ |u(x)-u(y)|^q}{d(x,y)^{sq}\mu(B(x,d(x,y)))}\,d\mu(y)\right)^\frac{1}{q}.
\end{split}
\end{equation}
Set
\begin{equation}\label{eq0_8.9}
    \psi(x)=(2^{i_0}r)^s\left(\int_{A_{i_0+2}}\frac{ |u(x)-u(y)|^q}{d(x,y)^{sq}\mu(B(x,d(x,y)))}\,d\mu(y)\right)^\frac{1}{q}.
\end{equation}
Since $|u(x)|\le \psi(x)$ for all $x$, Lemma~\ref{median prps} gives
\begin{align}\label{eq1_5.9}
    |m_{{u}}(A_{i_0+1})|&\le |m_{{\psi}}(A_{i_0+1})|\\
    &\le\left(\intav_{A_{i_0+1}}\psi^p(x)\,d\mu(x)\right)^\frac{1}{p}\\
    &\lesssim\frac{(2^{i_0}r)^s}{\mu(B_{i_0})^\frac{1}{p}}\left(\int_{A_{i_0+1}}\left(\int_{A_{i_0+2}}\frac{ |u(x)-u(y)|^q}{d(x,y)^{sq}\mu(B(x,d(x,y)))}\,d\mu(y)\right)^\frac{p}{q}\,d\mu(x)\right)^\frac{1}{p}.
\end{align}
By equations \eqref{eq1_24.6}, \eqref{eq1_5.9} and H\"older's inequality for $p>1$, we have
 \begin{align}\label{eq1_25.6}
1&\lesssim \left(\sum_{i=0}^{i_0}\left(\frac{(2^{i}r)^{sp}}{\mu(B_i)}\right)^{\frac{1}{p-1}}\right)^\frac{p-1}{p}\left(\sum_{i=0}^{\infty}\int_{A_i}\left(\int_{A_{i+1}}\frac{|u(x)-u(y)|^q}{d(x,y)^{sq}\mu(B(x,d(x,y)))}\,d\mu(y)\right)^\frac{p}{q}\,d\mu(x)\right)^\frac{1}{p}\\
&\leq\left(\sum_{i=0}^{i_0}\left(\frac{(2^{i}r)^{sp}}{\mu(B_i)}\right)^{\frac{1}{p-1}}\right)^\frac{p-1}{p}\left(\int_{X}\left(\int_{X}\frac{|u(x)-u(y)|^q}{d(x,y)^{sq}\mu(B(x,d(x,y)))}\,d\mu(y)\right)^\frac{p}{q}\,d\mu(x)\right)^\frac{1}{p}.
 \end{align}
 Therefore, by Lemma~\ref{Lemma 8.1} we have
 \begin{align}
\left(\sum_{i=0}^{i_0-1}\int_{2^{i}r}^{2^{i+1}r}\left(\frac{\hat{r}^{sp}}{\mu(B_{\hat{r}})}\right)^{\frac{1}{p-1}}\frac{d\hat{r}}{\hat{r}}\right)^{1-p } 
\lesssim \int_{X}\left(\int_{X}\frac{|u(x)-u(y)|^q}{d(x,y)^{sq}\mu(B(x,d(x,y)))}\,d\mu(y)\right)^{\frac{p}{q}}\,d\mu(x).
 \end{align}
Similarly for $p\leq 1$ and $\psi$ as in \eqref{eq0_8.9}, using \eqref{ineq-p<1} and the Lemma~\ref{median prps}, we obtain
\begin{align}\label{eq1_2.7}
1\le|m_{{u}}(A_0)|^p\le& |m_{{u}}(A_0)-m_{{u}}(A_{i_0+1})|^p+|m_{{u}}(A_{i_0+1})|^p\\
\le  &  \sum_{i=0}^{i_0}|m_{u}(A_{i})-m_{u}(A_{i+1})|^p+|m_{{\psi}}(A_{i_0+1})|^p\\
\lesssim & \sum_{i=0}^{i_0}\frac{(2^{i}r)^{sp}}{\mu(B_i)}\int_{A_i}\left(\int_{A_{i+1}}\frac{|u(x)-u(y)|^q}{d(x,y)^{sq}\mu(B(x,d(x,y)))}\,d\mu(y)\right)^\frac{p}{q}\,d\mu(x)\\
\lesssim & \sum_{i=0}^{i_0-1}\sup_{2^{i}r\le \hat{r}\le 2^{i+1}r}\frac{\hat{r}^{sp}}{\mu(B_{\hat{r}})}\int_{A_i}\left(\int_{A_{i+1}}\frac{|u(x)-u(y)|^q}{d(x,y)^{sq}\mu(B(x,d(x,y)))}\,d\mu(y)\right)^\frac{p}{q}\,d\mu(x)\\
\lesssim & \sup_{r\le \hat{r}<2\hat{r}\le R}\frac{\hat{r}^{sp}}{\mu(B_{\hat{r}})}\int_{X}\left(\int_{X}\frac{|u(x)-u(y)|^q}{d(x,y)^{sq}\mu(B(x,d(x,y)))}\,d\mu(y)\right)^\frac{p}{q}\,d\mu(x).
\end{align}
 Then, we have 
 \begin{align}
\inf_{r\le \hat{r}<2\hat{r}\le R}\frac{\mu(B_{\hat{r}})}{\hat{r}^{sp}}
\lesssim
\int_{X}\left(\int_{X}\frac{|u(x)-u(y)|^q}{d(x,y)^{sq}\mu(B(x,d(x,y)))}\,d\mu(y)\right)^\frac{p}{q}\,d\mu(x).
 \end{align}
Now, since $u$ is admissible for $\mathrm{cap}_{s,p,q}(B_r,B_{i_0+1})$ and $R\le 2^{i_0+1}r$, the claim follows by monotonicity.
\end{proof}

\begin{remark}
   Note that some restriction of the type $R<\frac{1}{4}\diam(X)$ cannot be removed. For example, if $X=[-R,R]$, then the constant function 1 is an admissible function for $\mathrm{cap}_{s,p}(B_r,B_R)$ and thus $\mathrm{cap}_{s,p}(B_r,B_R)=0$. 
\end{remark}

\begin{corollary}\label{lem0_4.9.26}
   Let $s\in(0,1)$, $p\in(0,\infty)$ and $q\in(0,\infty]$.   Then for $0<2r<\frac{1}{4}\diam(X)$,  we have
\begin{equation}\label{eq0_7.9}
\mathrm{cap}_{s,p,q}(B_r,B_{2r})
\gtrsim
\frac{\mu(B_r)}{r^{sp}}.
\end{equation}
\end{corollary}

\subsection{Upper bound for capacity}\label{UB estimate}
  Next, we establish the upper bound for the fractional capacity of fat annuli. Unlike the lower bound, the upper bound requires a controlled volume growth on the measure when $p<q$. This condition is needed to control the geometry at infinity and is automatically satisfied when $X$ is bounded.
  
  In particular, let $X$ be an unbounded metric space, and let $0<s<1$, $0<p< q\le\infty$. We consider the following \emph{dyadic condition}:
\begin{equation}\label{Rs,p,q finiteness}
     D_{s,p,q}(\mu)= \sum_{j=1}^\infty \frac{\mu(B(O,2^j))^{1-\frac{p}{q}}}{2^{jsp}}<\infty.\tag{D1}
\end{equation}
 This growth constraint is necessary as demonstrated in Lemma~\ref{necessary condition} below.
\begin{remark}
     Notice that for $p\ge q$, the condition above is always true by the reverse doubling property of $\mu$ \eqref{iteration reverse}, since $s>0\geq\frac{\sigma}{p}-\frac{\sigma}{q}$. The following argument shows that if $D_{s,p,q}(\mu)=\infty$, then there are no admissible functions for $\mathrm{cap}_{s,p,q}(E,\Omega)$ that belong to $\dot{W}^s_{p,q}$, and hence $\mathrm{cap}_{s,p,q}(E,\Omega)=\infty$.
 \end{remark}
\begin{lemma}\label{necessary condition}
    Let $0<s<1$, $0<p< q\le\infty$ and suppose that $ D_{s,p,q}(\mu)= \infty$. Let \(X\) be an unbounded metric space, and let $\Omega\subset X$ be a non-empty bounded open set. If $u$ is a non-trivial bounded measurable function such that $u=0$ on $\Omega^c$, then $u\notin \dot{W}^s_{p,q}(X)$.
\end{lemma}
\begin{proof}
Let  $u$ be a non-trivial bounded measurable function such that $u=0$ on $\Omega^c$. Set $K=\mathrm{supp}(u)$. Since $\Omega$ is bounded, there exists an integer $j_0\in\N$ such  that $K\subset B(O,2^{j_0})$ and $u=0$ on $B(O,2^{j_0})^c$. Fix $j_1\in\N$ such that $j_1>j_0$. \\
Then
\begin{equation}
    \begin{split}
\int_X&\left(\int_{X}\frac{|u(x)-u(y)|^q}{d(x,y)^{sq}\mu(B(x,d(x,y)))}\,d\mu(y)\right)^\frac{p}{q}\,d\mu(x)\\  
&\qquad\qquad\qquad\geq\int_{B(O,2^{j_1})^c}\left(\int_{K}\frac{|u(y)|^q}{d(x,y)^{sq}\mu(B(x,d(x,y)))}\,d\mu(y)\right)^\frac{p}{q}\,d\mu(x).
\end{split}
\end{equation}
Now, since $d(O,y)\le d(O,x)$ gives $d(x,y)\leq d(O,x)+d(O,y)\le2d(O,x)$, it follows from the above inequality that
\begin{equation}\label{m2}
    \begin{split}
\int_X&\left(\int_{X}\frac{|u(x)-u(y)|^q}{d(x,y)^{sq}\mu(B(x,d(x,y)))}\,d\mu(y)\right)^\frac{p}{q}\,d\mu(x)\\  
&\qquad\qquad\qquad
\gtrsim\int_{B(O,2^{j_1})^c}\left(\int_{K}\frac{|u(y)|^q}{d(O,x)^{sq}\mu(B(x,2d(O,x)))}\,d\mu(y)\right)^\frac{p}{q}\,d\mu(x).
\end{split}
\end{equation}
   Now since $d(O,x)<2d(O,x)$ we have $O\in B(x,2d(O,x))$ and hence by using \eqref{iteration doubling} we get
                    \begin{equation}\label{m0}
                        \frac{\mu(B(O,d(O,x)))}{\mu(B(x,2d(O,x)))}\geq c_Q\left(\frac{1}{2}\right)^Q.
                    \end{equation}
Therefore, by applying \eqref{m2} and \eqref{m0}, we get
               
\begin{equation}\label{m4}
    \begin{split}
\int_X&\left(\int_{X}\frac{|u(x)-u(y)|^q}{d(x,y)^{sq}\mu(B(x,d(x,y)))}\,d\mu(y)\right)^\frac{p}{q}\,d\mu(x)\\
&\qquad\qquad\gtrsim
\mu(K)^\frac{p}{q}\sum_{j\geq j_1}\int_{B(O,2^{j+1})\setminus B(O,2^j)}\frac{1}{d(O,x)^{sp}\mu(B(O,d(O,x))^\frac{p}{q}}\,d\mu(x)\\
&\qquad\qquad\gtrsim
\mu(K)^\frac{p}{q}\sum_{j\geq j_1}\frac{\mu(B(O,2^{j}))^{1-\frac{p}{q}}}{2^{jsp}}=\infty.
\end{split}
\end{equation}
   Hence, the claim follows.
\end{proof}

\begin{remark}
   Notice that when $X=\R^n$ and $\mu$ is a standard Lebesgue measure, the above lemma recovers \cite[Remark~3.8]{PrEe}. In fact, it was shown in \cite{AnPeKa} that $\dot{W}^s_{p,q}(\R^n)$ contains only constant functions when  $D_{s,p,q}(\mu)=\infty$.
\end{remark}

 We also require a stronger \emph{dyadic tail condition} on the measure when $p < q$. We impose this condition to ensure the resulting upper bound takes a specific form that matches our lower bound, establishing a sharp, two-sided capacitary inequality. 
If $D_{s,p,q}(\mu)<\infty$, a finite upper bound can still be obtained, but it may fail to match the lower bound. Our condition is therefore essential for obtaining a sharp bound for the fractional capacity. 
 
In particular, let $0<s<1$ and $0<p< q\le\infty$. Set $N=\left\lceil{\log_2(\mathrm{diam}(X))}\right\rceil$. We say that a measure $\mu$ satisfies the \emph{dyadic tail condition} if for all $r>0$ and $j_0\in\N$ $(j_0\le N)$, there exists a constant $C_{\mathrm{DT}}\ge1$ independent of $q$, $j_0$ and $r$ such that
     \begin{equation}\label{eq2_6.9}
     D_{s,p,q}(\mu,r,j_0)=  \sum_{j=j_0}^N \frac{\mu(B(O,2^jr))^{1-\frac{p}{q}}}{(2^jr)^{sp}}\le C_{\mathrm{DT}}^{1/q}\,\frac{\mu(B(O,2^{j_0}r))^{1-\frac{p}{q}}}{(2^{j_0}r)^{sp}},\tag{D2}
   \end{equation}
   where $C_{\mathrm{DT}}^{1/q}=1$ when $q=\infty$.
\begin{remark}
      Observe that when $p\ge q$, condition \eqref{eq2_6.9} follows from the reverse doubling property of $\mu$ \eqref{iteration reverse}. Note that if $s>\frac{Q}{p}-\frac{Q}{q}$, where $Q$ is as in \eqref{iteration doubling}, then \eqref{eq2_6.9} holds.
\end{remark}

The following lemma will be useful in the sequel.

\begin{lemma}\label{lem0_6.9}
    Let $0<s<1$ and $0<q\le p<\infty$. 
 Then, for $0<2r<\mathrm{diam}(X)$ we have
    \begin{equation}
        \mathrm{cap}_{s,p,q}(B_r,B_{2r})\lesssim\frac{\mu(B_r)}{r^{sp}},
    \end{equation}
    where $B_r, B_{2r}$ the comparison constant depends on $s$, $p$, and $C(c_D,c_{RD},C_{DT})^{1/q}$.

    Moreover, the above estimate also hold for $0<p<q\le\infty$, provided that the measure $\mu$ satisfies the dyadic tail condition \eqref{eq2_6.9}.
\end{lemma}
\begin{proof}
Define

\begin{equation}
        u(x)=\mathbf{1}_{B_r}+\left(1-\frac{2\mathrm{dist}(x,O)}{3r}\right)\mathbf{1}_{B_{2r}\setminus B_r}.
    \end{equation}
The function $u$ is $(2/3r)$-Lipschitz and admissible for $\mathrm{cap}_{s,p,q}(B_r,B_{2r})$.
 Then by Lemma~\ref{Lip func} we have
\[
\int_{B_{4r}}\left(\int_X
\frac{|u(x)-u(y)|^{q}}
     {d(x,y)^{s q}\mu(B(x,d(x,y)))}\,d\mu(y)\right)^\frac{p}{q}d\mu(x)\lesssim\frac{\mu(B_r)}{r^{sp}}.
\]
 Take $A_j=B(O,2^{j+1}r)\setminus B(O,2^jr)$ for all $j\in\N$. Then we have
\begin{align}
\int_{B_{4r}^c}&\left(\int_{B_{2r}}
\frac{|u(x)-u(y)|^{q}}
     {d(x,y)^{s q}\mu(B(x,d(x,y)))}\,d\mu(y)\right)^\frac{p}{q}d\mu(x)\\
     &\leq \sum_{j=2}^N \int_{A_j}\left(\int_{B_{2r}}
\frac{1}
     {d(x,y)^{s q}\mu(B(x,d(x,y)))}\,d\mu(y)\right)^\frac{p}{q}d\mu(x)\\
     &\lesssim \sum_{j=2}^N \int_{A_j}\left(
\frac{\mu(B_{2r})}{(2^jr)^{sq}\mu(B(x,2^jr))}\right)^\frac{p}{q}d\mu(x)\\
&\lesssim \sum_{j=2}^N \frac{\mu(B_{2r})^\frac{p}{q}\mu(B(O,2^{j}r))^{1-\frac{p}{q}}}{(2^jr)^{sp}}\\
&\lesssim \frac{\mu(B_{2r})}{r^{sp}},
\end{align}
where the last inequality follows from \eqref{eq2_6.9} when $p<q$, and from \eqref{iteration reverse} together with $s>0$ when $p\ge q$. Hence, the claim follows.
\end{proof}


\begin{corollary}\label{ps<1}
     Let $0<s<1$ and \(0<q\le p<\infty\).  Then for $0<2r\le R<\diam(X)$ we have
\begin{equation}\label{eq3-29.8}
\mathrm{cap}_{s,p,q}(B_r,B_R)
\lesssim
\inf_{r\le \hat{r}< 2\hat{r}\le R}\left(\frac{\hat{r}^{sp}}{\mu(B_{\hat{r}})}\right)^{-1 },
\end{equation}
where the comparison constant depends on $s$, $p$, and $C(c_D,c_{RD},C_{DT})^{1/q}$.

Moreover, the above estimates hold for $0<p<q\le\infty$, provided that the measure $\mu$ satisfies the dyadic tail condition \eqref{eq2_6.9}.
\end{corollary}
\begin{proof}
Let   $0<r\le \hat{r}< 2\hat{r}\le R$. Then by monotonicity and Lemma~\ref{lem0_6.9}, we have
    \begin{align}
        \mathrm{cap}_{s,p,q}(B_r,B_R)
        \leq \inf_{r\le \hat{r}< 2\hat{r}\le R}\mathrm{cap}_{s,p,q}(B_{\hat{r}},B_{2\hat{r}})
        \lesssim\inf_{r\le \hat{r}< 2\hat{r}\le R}\frac{\mu(B_{\hat{r}})}{\hat{r}^{sp}}.
    \end{align}
Hence, the claim follows.
\end{proof}

\subsection{Proof of Theorem~\ref{sharp capacity estimate annuli}}
We first establish the following theorem, which will be used in Theorem~\ref{sharp capacity estimate annuli}.
\begin{theorem}\label{thm0_1.9}
   Let $0<s<1$, $1<p<\infty$, $0<q\le p$, and let  $i_0\geq1$ be an integer. 
  Then, for $2^{i_0}r < \frac{1}{4}\mathrm{diam}(X)$,  we have
\begin{equation}\label{eq2-29.8}
\mathrm{cap}_{s,p,q}(B_0,B_{i_0})
\lesssim
\left(\sum_{i=1}^{i_0}
\mathrm{cap}_{s,p,q}(B_{i-1},B_{i})^{1/(1-p)}
\right)^{1-p}, 
\end{equation}
where $B_i=B(O,2^ir)$ for $i=0,1,\dots, i_0$. Here, the comparison constant depends on  $s$, $p$, and $C(c_D,c_{RD},C_{DT})^{1/q}$.

Moreover, the above estimate also holds for $0<p<q\le\infty$, provided that the measure $\mu$ satisfies the dyadic tail condition \eqref{eq2_6.9}.
\end{theorem}
\begin{proof}
 Let $i_0\ge 1$. For $i=0,1,\dots,i_0$, set $B_i=B(O,2^ir)$ and $B_{-i}=\emptyset$. 
Take \\$B_{i_0+1}= X\setminus B(O,2^{i_0}r)$.
Then under the given assumptions, Lemma~\ref{lem0_4.9.26} and ~\ref{lem0_6.9}, give that $0<\mathrm{cap}_{s,p,q}(B_{i-1},B_{i})<\infty$ for all $1\le i\le i_0$.

Let $\varepsilon>0$ be given. Then for each $i=1,\dots,i_0$, we choose an admissible function $u_{i-1}$ for $\mathrm{cap}_{s,p,q}(B_{i-1},B_{i})$ such that $u_{i-1}=0$ on $B_{i}^c$,  $u_{i-1}=1$ on $B_{i-1}$, and  
\begin{equation}\label{eq0_ 16.6.26}
    \int_{X}\left(\int_{X}\frac{|u_{i-1}(x)-u_{i-1}(y)|^q}{d(x,y)^{sq}\mu(B(x,d(x,y)))}\,d\mu(y)\right)^{\frac{p}{q}}\,d\mu(x)
\le
\mathrm{cap}_{s,p,q}(B_{i-1},B_{i})+\varepsilon .
\end{equation}
Set $u_{-i}\equiv 0$ for $i\in\N$. Let $\{a_i\}_{0\le i\le i_0-1}$ be a sequence of
nonnegative real numbers with
\[
\sum_{i=1}^{i_0} a_{i-1} = 1,
\]
and let $a_{-i}=0$.
Construct an admissible function $u$ for $\mathrm{cap}_{s,p,q}(B_0,B_{i_0})$ as
\begin{equation}\label{eq3_2.7}
    u=\sum_{i=1}^{i_0} a_{i-1} u_{i-1}.
\end{equation}
For this choice of $u$, we have
\begin{align}
    [u]_{W^s_{p,q}(X)}^p&\lesssim\int_{ B_{i_0}}\left(\int_{ B_{i_0}}\frac{|u(x)-u(y)|^q}{d(x,y)^{sq}\mu(B(x,d(x,y)))}\,d\mu(y)\right)^{\frac{p}{q}}\,d\mu(x)\\
    &\quad+\int_{ B_{i_0}}\left(\int_{ B_{i_0+1}}\frac{|u(x)-u(y)|^q}{d(x,y)^{sq}\mu(B(x,d(x,y)))}\,d\mu(y)\right)^{\frac{p}{q}}\,d\mu(x)\\
    &\qquad+\int_{ B_{i_0+1}}\left(\int_{ B_{i_0}}\frac{|u(x)-u(y)|^q}{d(x,y)^{sq}\mu(B(x,d(x,y)))}\,d\mu(y)\right)^{\frac{p}{q}}\,d\mu(x)\\
     &=:I_1+I_2+I_3.
\end{align}
We estimate $I_1$ as follows:
\begin{align}
    I_1&\lesssim\sum_{0\leq i \leq i_0}\int_{B_i\setminus B_{i-1}}\left(\int_{ B_{i-2}}\frac{|u(x)-u(y)|^q}{d(x,y)^{sq}\mu(B(x,d(x,y)))}\,d\mu(y)\right)^{\frac{p}{q}}\,d\mu(x)\\
    &\quad+ \sum_{0\leq i \leq i_0-1}\int_{B_i\setminus B_{i-1}}\left(\sum_{i-1\leq j \leq i+1}\int_{B_j\setminus B_{j-1}}\frac{|u(x)-u(y)|^q}{d(x,y)^{sq}\mu(B(x,d(x,y)))}\,d\mu(y)\right)^{\frac{p}{q}}\,d\mu(x)\\
     &\qquad+\int_{B_{i_0}\setminus B_{i_0-1}}\left(\sum_{i_0-1\leq j \leq i_0}\int_{B_j\setminus B_{j-1}}\frac{|u(x)-u(y)|^q}{d(x,y)^{sq}\mu(B(x,d(x,y)))}\,d\mu(y)\right)^{\frac{p}{q}}\,d\mu(x)\\
    &\qquad\quad+ \sum_{0\leq i \leq i_0-2}\int_{B_i\setminus B_{i-1}}\left(\sum_{i+2\leq j \le i_0}\int_{B_j\setminus B_{j-1}}\frac{|u(x)-u(y)|^q}{d(x,y)^{sq}\mu(B(x,d(x,y)))}\,d\mu(y)\right)^{\frac{p}{q}}\,d\mu(x)\\
     &=:I_{1,1}+I_{1,2}+I_{1,3}+I_{1,4}.
\end{align}
We estimate $I_{1,1}$ as follows:
\begin{align}
     I_{1,1}
     &\lesssim\sum_{0\leq i \leq i_0}\int_{B_i\setminus B_{i-1}}\left(\sum_{0\leq j \leq i-2}\int_{B_j\setminus B_{j-1}}\frac{\left(\sum_{k=j}^ia_{k-1}\right)^q}{(2^ir)^{sq}\mu(B(x,2^ir))}\,d\mu(y)\right)^{\frac{p}{q}}\,d\mu(x)\\
     &\lesssim\sum_{0\leq i \leq i_0}\frac{\mu(B_i)^{1-\frac{p}{q}}}{(2^ir)^{sp}}\left(\sum_{0\leq j \leq i-2}\mu(B_j\setminus B_{j-1})\left(\sum_{k=j}^ia_{k-1}\right)^q\right)^{\frac{p}{q}}.
\end{align}
Next, we estimate $I_{1,2}$ as follows. Recall that
\begin{equation}
    I_{1,2}= \sum_{0\leq i \leq i_0-1}\int_{B_i\setminus B_{i-1}}\left(\sum_{i-1\leq j \leq i+1}\int_{B_j\setminus B_{j-1}}\frac{|u(x)-u(y)|^q}{d(x,y)^{sq}\mu(B(x,d(x,y)))}\,d\mu(y)\right)^{\frac{p}{q}}\,d\mu(x).
\end{equation}
First, consider the case $j=i-1$ in the above integral. Then
\begin{align}
    \sum_{0\leq i \leq i_0-1}&\int_{B_i\setminus B_{i-1}}\left(\int_{B_{i-1}\setminus B_{i-2}}\frac{|u(x)-u(y)|^q}{d(x,y)^{sq}\mu(B(x,d(x,y)))}\,d\mu(y)\right)^{\frac{p}{q}}\,d\mu(x)\\
   & \lesssim \sum_{0\leq i \leq i_0-1}a_{i-2}^p\int_{B_i\setminus B_{i-1}}\left(\int_{B_{i-1}\setminus B_{i-2}}\frac{|u_{i-2}(x)-u_{i-2}(y)|^q}{d(x,y)^{sq}\mu(B(x,d(x,y)))}\,d\mu(y)\right)^{\frac{p}{q}}\,d\mu(x)\\
    &\qquad+ \sum_{0\leq i \leq i_0-1}a_{i-1}^p\int_{B_i\setminus B_{i-1}}\left(\int_{B_{i-1}\setminus B_{i-2}}\frac{|u_{i-1}(x)-u_{i-1}(y)|^q}{d(x,y)^{sq}\mu(B(x,d(x,y)))}\,d\mu(y)\right)^{\frac{p}{q}}\,d\mu(x).
\end{align}
Repeating the similar argument for $j\in\{i,i+1\}$, we obtain
\begin{align}
   I_{1,2}&\lesssim \sum_{0\leq i \leq i_0-1}a_{i-2}^p\int_{B_i\setminus B_{i-1}}\left(\int_{B_{i-1}\setminus B_{i-2}}\frac{|u_{i-2}(x)-u_{i-2}(y)|^q}{d(x,y)^{sq}\mu(B(x,d(x,y)))}\,d\mu(y)\right)^{\frac{p}{q}}\,d\mu(x)\\
         &\quad+ \sum_{0\leq i \leq i_0-1}a_{i-1}^p\int_{B_i\setminus B_{i-1}}\left(\int_{B_{i+1}\setminus B_{i-2}}\frac{|u_{i-1}(x)-u_{i-1}(y)|^q}{d(x,y)^{sq}\mu(B(x,d(x,y)))}\,d\mu(y)\right)^{\frac{p}{q}}\,d\mu(x)\\
         &\qquad+ \sum_{0\leq i \leq i_0-1}a_{i}^p\int_{B_i\setminus B_{i-1}}\left(\int_{B_{i+1}\setminus B_{i}}\frac{|u_{i}(x)-u_{i}(y)|^q}{d(x,y)^{sq}\mu(B(x,d(x,y)))}\,d\mu(y)\right)^{\frac{p}{q}}\,d\mu(x).
\end{align}
For $I_{1,3}$, we similarly obtain 
\begin{align}
    I_{1,3}&\lesssim a_{i_0-2}^p\int_{B_{i_0}\setminus B_{i_0-1}}\left(\int_{B_{i_0-1}\setminus B_{i_0-2}}\frac{|u_{i_0-2}(x)-u_{i_0-2}(y)|^q}{d(x,y)^{sq}\mu(B(x,d(x,y)))}\,d\mu(y)\right)^{\frac{p}{q}}\,d\mu(x)\\
    &\quad +  a_{i_0-1}^p\int_{B_{i_0}\setminus B_{i_0-1}}\left(\int_{B_{i_0}\setminus B_{i_0-2}}\frac{|u_{i_0-1}(x)-u_{i_0-1}(y)|^q}{d(x,y)^{sq}\mu(B(x,d(x,y)))}\,d\mu(y)\right)^{\frac{p}{q}}\,d\mu(x).
\end{align}
We now estimate for $I_{1,4}$ as follows:
\begin{align}
    I_{1,4}&\lesssim\sum_{0\le i\le i_0-2}\int_{B_i\setminus B_{i-1}}\left(\sum_{i+2\le j\le i_0}\int_{B_j\setminus B_{j-1}}\frac{\left(\sum_{k=i}^{j}a_{k-1}\right)^q}{(2^jr)^{sq}\mu(B(x,2^jr))}d\mu(y)\right)^\frac{p}{q}d\mu(x)\\
    &\lesssim\sum_{0\le i\le i_0-2}\mu({B_i\setminus B_{i-1}})\left(\sum_{i+2\le j\le i_0}\frac{\left(\sum_{k=i}^{j}a_{k-1}\right)^q}{(2^jr)^{sq}}\right)^\frac{p}{q}.
\end{align}
For $I_2$, we have 
\begin{align}
    I_2&\lesssim a_{i_0-1}^p\int_{B_{i_0}\setminus B_{i_0-1}}\left(\int_{ B_{i_0+1}}\frac{|u_{i_0-1}(x)-u_{i_0-1}(y)|^q}{d(x,y)^{sq}\mu(B(x,d(x,y)))}\,d\mu(y)\right)^{\frac{p}{q}}\,d\mu(x)\\
    &\quad+\sum_{0\leq i\leq i_0-1}\int_{ B_{i}\setminus B_{i-1}}\left(\int_{ B_{i_0+1}}\frac{|u(x)|^q}{d(x,y)^{sq}\mu(B(x,d(x,y)))}\,d\mu(y)\right)^{\frac{p}{q}}\,d\mu(x)\\
    &=:I_{2,1}+I_{2,2}.
\end{align}
We now estimate $I_{2,2}$ as follows:
\begin{align}
  I_{2,2}   &\lesssim\sum_{0\leq i\leq i_0-1} \int_{ B_{i}\setminus B_{i-1}}
   \left(\sum_{j= i_0+1}^N\int_{ B(O,2^{j}r)\setminus B(O,2^{j-1}r)}\frac{\left(\sum_{k=i}^{i_0}a_{k-1}\right)^q}{(2^{j}r)^{sq}\mu(B(O,2^{j}r))}\,d\mu(y)\right)^{\frac{p}{q}}\,d\mu(x)\\
    &\lesssim\sum_{0\leq i\leq i_0-1}\mu( B_{i}\setminus B_{i-1})\left(\frac{\left(\sum_{k=i}^{i_0}a_{k-1}\right)^q}{(2^{i_0}r)^{sq}}\right)^\frac{p}{q}.
\end{align}
For $I_3$, we have 
\begin{align}
     I_3&\lesssim a_{i_0-1}^p\int_{ B_{i_0+1}}\left(\int_{B_{i_0}\setminus B_{i_0-1}}\frac{|u_{i_0-1}(x)-u_{i_0-1}(y)|^q}{d(x,y)^{sq}\mu(B(x,d(x,y)))}\,d\mu(y)\right)^{\frac{p}{q}}\,d\mu(x)\\
    &\quad+
    \int_{ B_{i_0+1}}\left( \int_{ B_{i_0-1}}\frac{|u(x)-u(y)|^q}{d(x,y)^{sq}\mu(B(x,d(x,y)))}\,d\mu(y)\right)^{\frac{p}{q}}\,d\mu(x)\\
    &=:I_{3,1}+I_{3,2}.
\end{align}
We now estimate $I_{3,2}$ as follows:
\begin{align}
I_{3,2}&\lesssim\sum_{i=i_0}^N\int_{ B(O,2^{i+1}r)\setminus B(O,2^{i}r)}\left(\sum_{0\le j\le i_0-1}\int_{ B_j\setminus B_{j-1}}\frac{(\sum_{k=j}^{i_0}a_{k-1})^q}{(2^{i}r)^{sq}\mu(B(O,2^{i}r))}\,d\mu(y)\right)^{\frac{p}{q}}\,d\mu(x)\\
&\lesssim\sum_{i=i_0}^N\frac{\mu( B(O,2^{i}r))^{1-\frac{p}{q}}}{(2^{i}r)^{sp}}\left(\sum_{0\le j\le i_0-1}\mu(B_j\setminus B_{j-1})\left(\sum_{k=j}^{i_0}a_{k-1}\right)^q\right)^{\frac{p}{q}}\\
 &\lesssim \frac{\mu( B_{i_0})^{1-\frac{p}{q}}}{(2^{i_0}r)^{sp}}\left(\sum_{0\le j\le i_0-1}\mu(B_j\setminus B_{j-1})\left(\sum_{k=j}^{i_0}a_{k-1}\right)^q\right)^{\frac{p}{q}},
\end{align}
 where the last inequality follows from \eqref{eq2_6.9} when $p<q$, and from \eqref{iteration reverse} together with $s>0$ when $p\ge q$.

Let $\delta>0$ and $p>q$. Then by combining $I_{1,4}$ and $I_{2,2}$ we obtain 
\begin{align}
    I_{1,4}+I_{2,2} &\lesssim\sum_{0\le i\le i_0-1}\mu({B_i\setminus B_{i-1}})\left(\sum_{i+1\le j\le i_0}\frac{2^{-j\delta}\left(\sum_{k=i}^{j}a_{k-1}\right)^q}{(2^{j})^{sq-\delta} r^{sq}}\right)^\frac{p}{q}\\
     &\lesssim\sum_{0\le i\le i_0-1}\mu({B_i\setminus B_{i-1}})2^{-i\delta\frac{p}{q}}\sum_{i+1\le j\le i_0}\frac{\left(\sum_{k=i}^{j}a_{k-1}\right)^p}{(2^{j})^{sp-\delta\frac{p}{q}} r^{sp}},
\end{align}
using H\"older's inequality with exponent $p/q$. Fix $i$ and let $l=j-i$; here $j>i$. Then, $l=1,2,3,\dots,i_0-i$ and 
\begin{align}
    I_{1,4}+I_{2,2} &\lesssim\sum_{0\le i\le i_0-1}\mu({B_i\setminus B_{i-1}})2^{-i\delta\frac{p}{q}}\sum_{1\le l\le i_0-i}\frac{\left(\sum_{n=-1}^{l-1}a_{i+n}\right)^p}{(2^{l+i})^{sp-\delta\frac{p}{q}} r^{sp}}\\
    &=\sum_{0\le i\le i_0-1}\mu({B_i\setminus B_{i-1}})(2^ir)^{-sp}\sum_{1\le l\le i_0-i}\frac{\left(\sum_{n=-1}^{l-1}a_{i+n}\right)^p}{2^{lsp} }.
\end{align}
Applying Lemma~\ref{discrete weighted Hardy} with $u_n=v_n=2^{-nsp}$ for $1\le n\le i_0-i$ and $u_n=v_n=0$ otherwise, we obtain
\begin{equation}
    \sum_{n=k}^\infty u_n\lesssim 2^{-ksp}\text{ and }\sum_{n=0}^k v_n^{1-p'}\lesssim2^{\frac{ksp}{p-1}}.
\end{equation}
Thus, we have $A \le C(n,s,p)<\infty$ in \eqref{eq0_3.9}, which gives 
\begin{equation}
    I_{1,4}+I_{2,2}  \lesssim\sum_{0\le i\le i_0-1}\mu({B_i\setminus B_{i-1}})(2^ir)^{-sp}\sum_{1\le l\le i_0-i}\frac{a_{i+l-1}^p}{2^{lsp}},
\end{equation}
since $p>1$. Fix $i$ and take $m=i+l$. Thus, it follows from the above inequality and Tonelli's theorem that
\begin{align}
    I_{1,4}+I_{2,2}  &\lesssim\sum_{0\le i\le i_0-1}\mu({B_i\setminus B_{i-1}})\sum_{m=i+1}^{i_0}a_{m-1}^p(2^{m}r)^{-sp}\\
    &=\sum_{m=1}^{i_0}a_{m-1}^p(2^{m}r)^{-sp}\sum_{i=0}^{m-1}\mu({B_i\setminus B_{i-1}})\\
    &\leq \sum_{m=1}^{i_0}a_{m-1}^p(2^{m}r)^{-sp}\mu({B_m}).
\end{align}
Observe that the above inequality also holds for $p\leq q$, by using \eqref{ineq-p<1} in place of H\"older's inequality.

Finally, we will combine $I_{1,1}$ and $I_{3,2}$ to establish the remaining estimate. Let $p>q$ and $0<\delta_0<\sigma/q$, where $\sigma$ is as in \eqref{iteration reverse}. Using H\"older's inequality with exponent $p$, \eqref{ineq-p<1} and Tonelli's theorem, we obtain 
\begin{align}\label{eq1_30.8}
   I_{1,1}+I_{3,2}&\lesssim \sum_{0\leq i \leq i_0}\frac{\mu(B_i)^{1-\frac{p}{q}}}{(2^ir)^{sp}}\left(\sum_{0\leq j \leq i-1}\mu(B_j\setminus B_{j-1})\left(\sum_{k=j}^ia_{k-1}\right)^q\right)^{\frac{p}{q}}\\
   &\lesssim   \sum_{0\leq i \leq i_0}\frac{\mu(B_i)^{1-\frac{p}{q}}}{(2^ir)^{sp}}\left(\sum_{0\leq j \leq i-1}\mu(B_j\setminus B_{j-1})2^{-j\delta_0 q}\left(\sum_{k=j}^ia_{k-1}^p2^{k\delta_0 p}\right)^\frac{q}{p}\right)^{\frac{p}{q}}\\
    &\leq   \sum_{0\leq i \leq i_0}\frac{\mu(B_i)^{1-\frac{p}{q}}}{(2^ir)^{sp}}\left(\sum_{0\leq j \leq i-1}\mu(B_j\setminus B_{j-1})2^{-j\delta_0 q}\sum_{k=j}^ia_{k-1}^q2^{k\delta_0 q}\right)^{\frac{p}{q}}\\
    &\le  \sum_{0\leq i \leq i_0}\frac{\mu(B_i)^{1-\frac{p}{q}}}{(2^ir)^{sp}}\left(\sum_{0\leq k \leq i}a_{k-1}^q2^{k\delta_0 q}\sum_{j=0}^k\mu(B_j\setminus B_{j-1})2^{-j\delta_0 q}\right)^{\frac{p}{q}}\\
    &\lesssim   \sum_{0\leq i \leq i_0}\frac{\mu(B_i)^{1-\frac{p}{q}}}{(2^ir)^{sp}}\left(\sum_{0\leq k \leq i}a_{k-1}^q2^{k(\delta_0 q-\sigma)}\mu(B_k\setminus B_{k-1})\sum_{j=0}^k2^{j(\sigma-\delta_0 q)}\right)^{\frac{p}{q}}\\
    &\lesssim \sum_{0\leq i \leq i_0}\frac{\mu(B_i)^{1-\frac{p}{q}}}{(2^ir)^{sp}}\left(\sum_{0\leq k \leq i}a_{k-1}^q\mu(B_k\setminus B_{k-1})\right)^{\frac{p}{q}}.
   \end{align}
Applying Lemma~\ref{discrete weighted Hardy} with $u_n=v_n=\frac{\mu(B_n)^{1-\frac{p}{q}}}{(2^nr)^{sp}}$ for $0\le n\le i_0$ and $u_n=v_n=0$ otherwise, we obtain 
\begin{equation}
    \sum_{n=k}^\infty u_n\lesssim \frac{\mu(B_k)^{1 - \frac{p}{q}}}{(2^kr)^{sp}}\text{ and }\sum_{n=0}^k v_n^{1-(p/q)'}\lesssim \mu(B_k)(2^kr)^{\frac{spq}{p-q}},
\end{equation}
using \eqref{iteration reverse}. Thus, we have $A \le C(n,s,p,q)<\infty$ in \eqref{eq0_3.9}, which gives 
\begin{equation}\begin{split}
   I_{1,1}+I_{3,2}&\lesssim \sum_{0\leq i \leq i_0}\frac{\mu(B_i)^{1-\frac{p}{q}}}{(2^ir)^{sp}}a_{i-1}^p\mu(B_i)^\frac{p}{q}\\
   &=\sum_{0\leq i \leq i_0}\frac{\mu(B_i)}{(2^ir)^{sp}}a_{i-1}^p.
\end{split}
 \end{equation}
For $p\le q$, the estimate follows similarly. Fix $0<\delta_0<\sigma/q$. We use \eqref{ineq-p<1} and \eqref{eq2_6.9} to obtain
\begin{equation}\begin{split}
    I_{1,1}+I_{3,2}&\lesssim\sum_{0\leq i \leq i_0}\frac{\mu(B_i)^{1-\frac{p}{q}}}{(2^ir)^{sp}}\sum_{0\leq k \leq i}a_{k-1}^p\mu(B_k\setminus B_{k-1})^{\frac{p}{q}}\\
    &\leq\sum_{0\leq k \leq i_0}a_{k-1}^p\mu(B_k\setminus B_{k-1})^{\frac{p}{q}}\sum_{k\leq i \leq i_0}\frac{\mu(B_i)^{1-\frac{p}{q}}}{(2^ir)^{sp}}\\
      &\lesssim\sum_{0\leq k \leq i_0}a_{k-1}^p\frac{\mu(B_k)}{(2^kr)^{sp}}.
    \end{split}   
\end{equation}

Finally, combining the estimates for $I_{1}$, $I_{2}$ and $I_3$, we get
\begin{align}
    [u]_{W^s_{p,q}(X)}^p&\lesssim \sum_{1\leq i \leq i_0}a_{i-1}^p\frac{\mu(B_i)}{(2^ir)^{sp}}\\
     &\qquad+\sum_{1\leq i \leq i_0}a_{i-1}^p\int_{X}\left(\int_{X}\frac{|u_{i-1}(x)-u_{i-1}(y)|^q}{d(x,y)^{sq}\mu(B(x,d(x,y)))}\,d\mu(y)\right)^{\frac{p}{q}}\,d\mu(x)\\
     &\lesssim  \sum_{1\leq i \leq i_0}a_{i-1}^p\mathrm{cap}_{s,p,q}(B_{i-1},B_i)+\varepsilon,
     \end{align}
where the last inequality is obtained from \eqref{eq0_ 16.6.26} and Corollary~\ref{lem0_4.9.26}. Letting $\varepsilon\to 0$ and choosing
\begin{equation}
    a_i= \mathrm{cap}_{s,p,q}(B_i,B_{i+1} )^\frac{1}{1-p}\left(\sum_{1\leq i \leq i_0}\mathrm{cap}_{s,p,q}(B_i,B_{i+1} )^\frac{1}{1-p}\right)^{-1},
\end{equation}
we obtain the desired estimate.
\end{proof}



\begin{proof}[Proof of Theorem~\ref{sharp capacity estimate annuli}]
The lower bound for the capacity estimates follows from Lemma~\ref{UB capacity_ball}. When $p\le 1$, the corresponding upper bound follows from Corollary~\ref{ps<1}. 
 
Let $p>1$ and $i_0\geq 1$ be the smallest integer such that $2^{i_0}r \leq R< 2^{i_0+1}r$. For $i=1,\dots ,i_0$ we set    $B_i= B(O,2^ir)$. By Theorem~\ref{thm0_1.9} and Lemma~\ref{lem0_6.9}, we have
 \begin{align}
     \mathrm{cap}_{s,p,q}(B_r, B_R)&\le \mathrm{cap}_{s,p,q}(B_r, B_{i_0})\\
     &\lesssim \left(\sum_{i=1}^{i_0}
\mathrm{cap}_{s,p,q}(B_{i-1},B_{i})^{1/(1-p)}
\right)^{1-p} \\
 &\lesssim \left(\sum_{i=1}^{i_0}
\left(\frac{\mu(B_{i-1})}{(2^{i-1}r)^{sp}}\right)^{1/(1-p)}\right)^{1-p} \\
&\lesssim\left(\int_{r}^{R}\left(\frac{\hat{r}^{sp}}{\mu(B_{\hat{r}})}\right)^{\frac{1}{p-1}}\frac{d\hat{r}}{\hat{r}}\right)^{1-p },
 \end{align}
       where the last inequality follows from Lemma~\ref{Lemma 8.1} and doubling property of $\mu$. Hence, the claim follows. 
\end{proof}
\begin{proof}[Proof of Proposition~\ref{prop0_13.9}]
   The properties (i)-(iii) are proved separately.
   \begin{itemize}
       \item [(i)] Let $p>1$ and $0<2r<R$. By Theorem~\ref{sharp capacity estimate annuli}, we have
    \begin{align}
        \mathrm{cap}_{s,p,q}(\{O\}, B_R)\leq \limsup_{r\to0}\mathrm{cap}_{s,p,q}(B_r, B_R)
\lesssim\left(\int_{0}^{R}\left(\frac{\hat{r}^{sp}}{\mu(B_{\hat{r}})}\right)^{\frac{1}{p-1}}\frac{d\hat{r}}{\hat{r}}\right)^{1-p }=0.
    \end{align}
 A similar argument gives the result for $p\le1$.  Hence, the claim follows.\\
 \item[(ii)]  Let $C$ be a compact subset of $X$ and $r>0$ be such that $C\subset B(O,r)$. By Theorem~\ref{sharp capacity estimate annuli}, we have
    \begin{align}
        \mathrm{cap}_{s,p,q}(C, X)
\lesssim\left(\int_{r}^{\infty}\left(\frac{\hat{r}^{sp}}{\mu(B_{\hat{r}})}\right)^{\frac{1}{p-1}}\frac{d\hat{r}}{\hat{r}}\right)^{1-p }=0.
    \end{align}
  A similar argument gives the result for $p\le1$.  Hence, the claim follows.\\
  \item[(iii)] The proof follows immediately from Theorem~\ref{sharp capacity estimate annuli}.
 \end{itemize}
\end{proof}

\section{Capacity estimates for thin annuli}\label{thin annuli section}
In this section, we establish two-sided bounds for the capacity of thin
annuli. A lower bound for thin annuli requires an additional assumption on
the underlying space, which we state in Lemma~\ref{thin annuli_lb}.

\subsection{Lower bound for the capacity}
\begin{lemma}\label{thin annuli_lb}
     Let $s\in(0,1)$, $p\in(0,\infty)$ and $q\in(0,\infty]$. Let $0<r<\frac{1}{4}\diam(X)$, $0<\varepsilon<1$, $\lambda>1$ with $\lambda \varepsilon<1$  and $E\subset B_{r+\lambda\varepsilon r}\setminus B_{r+\varepsilon r}$ is such that $\mu(E)\geq c\mu(B_{r+\lambda \varepsilon r}\setminus B_{r+\varepsilon r})$ for some $c\in (0,1]$.  In addition, assume that there exists $\lambda'>1$ such that for every $x\in E $ there exists a point $x'$ such that $B(x',\varepsilon r/2)\subset B(x,\lambda'\varepsilon r)\cap B_r$.
 Then
   \begin{equation}\label{eq2_5.10}
\mathrm{cap}_{s,p,q}(B_r,B_{r+\varepsilon r})
\gtrsim
\frac{\mu(B_{r+\lambda\varepsilon r}\setminus B_{r+\varepsilon r})}{(\varepsilon r)^{sp}}+\frac{\mu(B_r)}{r^{sp}},
\end{equation}
where the comparison constant depends on $\lambda'$, $c$, $s$, $p$, and $C(c_{RD},c_D)^{1/q}$. 
\end{lemma}
\begin{proof}
Let $u$ be any admissible function for $\mathrm{cap}_{s,p,q}(B_r,B_{r+\varepsilon r}).$ 
Then
\begin{align}
    \int_{E}&\left(\int_{B(x',\varepsilon r/2)}\frac{|u(x)-u(y)|}{d(x,y)^{sq}\mu(B(x,d(x,y)))}d\mu(y)\right)^\frac{p}{q}\,d\mu(x)\\
    &\gtrsim \int_{E}\left(\int_{B(x',\varepsilon r/2)}\frac{1}{(\varepsilon r)^{sq}\mu(B(x,(\lambda'+1/2)\varepsilon r))}d\mu(y)\right)^\frac{p}{q}\,d\mu(x)\\
    &\gtrsim\frac{\mu(B_{r+\lambda\varepsilon r}\setminus B_{r+\varepsilon r})}{(\varepsilon r)^{sp}}.
    \end{align}

Set $A_j=B(O,2^{j+1}r)\setminus B(O,2^jr)$ for all $j\in\N$. Then
\begin{align}
    \int_{B_{r}}&\left(\int_{B_{2r}^c}\frac{|u(x)-u(y)|}{d(x,y)^{sq}\mu(B(x,d(x,y)))}d\mu(y)\right)^\frac{p}{q}\,d\mu(x)\\
    &\gtrsim \int_{B_{r}}\left(\sum_{j=1}^N\int_{A_j}\frac{1}{(2^jr)^{sq}\mu(B(x,2^jr))}d\mu(y)\right)^\frac{p}{q}\,d\mu(x)\\
    &\gtrsim \frac{\mu(B_r)}{r^{sp}}.
    \end{align}
Hence, the claim follows.
\end{proof}

\begin{remark} \label{weak assu_thin}
If $X$ is a length space with $\mu$ doubling, the assumptions of the above lemma always hold. Indeed, for $x\in E$ one can pick
$x'$ on a curve from $O$ to $x$ of length less than $d(O,x)+\varepsilon r$ with
$r-c\varepsilon r\le d(O,x')\le r-\varepsilon r/2$, for some $c\in[1/2,1]$. This is not true for general connected spaces.
\end{remark}

\subsection{Two-sided bounds}

To obtain a sharp two-sided estimate for the capacity of thin annuli, we require the following annular decay property.

Recall that $O\in X$ is a fixed point and that we write $B_r=B(O,r)$.
Then, we say that $\mu$ satisfies the  upper \emph{$\eta$-annular decay property} if there exist $c_A\geq 1$ and $\eta>0$ 
such that for all $r>0$ and all $0\leq\delta_1\le\delta_2\le 1$, 
 \begin{equation}\label{eq3_5.10}
 \begin{split}
       \mu(B_{r-\delta_1 r}\setminus B_{r-\delta_2 r})&\leq c_A\, |\delta_1-\delta_2|^{\eta}\mu(B_r).
 \end{split}
   \end{equation}
It satisfies the reverse \emph{$\eta$-annular decay property} if the same shells satisfy the inequalities with $c_A^{-1}$ in place of $c_A$ and $\geq$ in place of $\leq$. It has the \emph{$\eta$-annular decay property} if both hold.\\

The global upper annular decay property was first studied for manifolds by Colding--Minicozzi~\cite{CoMi} under the name volume regularity
property manifolds. Buckley~\cite{BuSt} independently studied it for general metric spaces, calling it the annular decay property. Its pointwise
version has been used in the study of the first-order Sobolev spaces
$N^{1,p}$ in~\cite{BBJu}.

In the Sobolev setting, the reverse $\eta$-annular decay property can be obtained
for $\eta\ge 1$ if $\mu$ is doubling and reverse doubling and $X$ supports an
$\eta$-Poincar\'e inequality, see~\cite[Lemma~5.2]{BBJu}. The proof relies on the
locality of the gradient, which vanishes where the test function is constant, so
that the energy is concentrated in the annulus. In the fractional setting, the Gagliardo energy is nonlocal and pointwise in nature, and hence cannot be localized to the annulus in the same way.
We therefore assume the reverse $\eta$-annular decay property in Theorem~\ref{thin annuli_main} to obtain
a two-sided sharp bound for the capacity of thin annuli. For $X=\mathbb{R}^n$
with the Lebesgue measure, it holds with $\eta=1$.


We now establish an upper bound estimate for the capacity of thin annuli.

\begin{lemma}\label{thin annuli_ub}
     Let $0<s<1$, $0<q\le p<\infty$ and $0<\varepsilon<1$. Assume  $0<r<\frac{1}{4}\diam(X)$ and  that the measure $\mu$ satisfies the upper $\eta$-annular decay property. Then   
     \begin{equation}
\mathrm{cap}_{s,p,q}(B_r,B_{r+\varepsilon r})
\lesssim
 \begin{cases}
   \frac{\mu(B_r)}{r^{sp}}\left(1+\varepsilon^{\eta-sp}+\varepsilon^{-sp}\right), & sp\ge\eta,\\ 
    \frac{\mu(B_r)}{r^{sp}}( \varepsilon^{\eta-sp}+1)\,\quad\qquad,& sp< \eta,
\end{cases}
\end{equation}
where the comparison constant depends on $s$, $p$, $\eta$, $c_A$, and $C(c_D,c_{RD},C_{DT})^{1/q}$. 

Moreover, the above estimates also hold for $0<p<q\le\infty$, provided that the measure $\mu$ satisfies the dyadic tail condition \eqref{eq2_6.9}.
\end{lemma}
\begin{proof}
Define 
\begin{equation}
    u(x)=\left(1-\frac{\dist(x,B_r)}{\varepsilon r}\right)\1_{B_{r+\varepsilon r}}.
\end{equation}
The function $u$ is $(1/\varepsilon r)$-Lipschitz and admissible for $\mathrm{cap}_{s,p,q}(B_r,B_{r+\varepsilon r})$.
Set $A_j=B(O,2^{j+1}r)\setminus B(O,2^jr)$ for all $j\in\N$. Then we have
\begin{align}
\int_{B_{2r}^c}&\left(\int_{B_{r+\varepsilon r}}
\frac{|u(x)-u(y)|^{q}}
     {d(x,y)^{s q}\mu(B(x,d(x,y)))}\,d\mu(y)\right)^\frac{p}{q}d\mu(x)\\
     &\leq \sum_{j=1}^N \int_{A_j}\left(\int_{B_{r+\varepsilon r}}
\frac{1}
     {d(x,y)^{s q}\mu(B(x,d(x,y)))}\,d\mu(y)\right)^\frac{p}{q}d\mu(x)\\
     &\lesssim \sum_{j=1}^N \int_{A_j}\left(
\frac{\mu(B_{2r})}{(2^jr)^{sq}\mu(B(x,2^jr))}\right)^\frac{p}{q}d\mu(x)\\
&\lesssim \sum_{j=1}^N \frac{\mu(B_{2r})^\frac{p}{q}\mu(B(O,2^{j}r))^{1-\frac{p}{q}}}{(2^jr)^{sp}}\\
&\lesssim \frac{\mu(B_{2r})}{r^{sp}},
\end{align}
where the last inequality follows from \eqref{eq2_6.9} when $p<q$, and from \eqref{iteration reverse} together with $s>0$ when $p\ge q$. 
Now since $u$ is Lipschitz, we have by Lemma~\ref{Lip func},
\[
\int_{B_{r+\varepsilon r}\setminus B_{r}}\left(\int_X
\frac{|u(x)-u(y)|^{q}}
     {d(x,y)^{s q}\mu(B(x,d(x,y)))}\,d\mu(y)\right)^\frac{p}{q}d\mu(x)\lesssim\frac{\mu(B_{r+\varepsilon r}\setminus B_{r})}{(\varepsilon r)^{sp}}\lesssim\varepsilon^{\eta-sp}\frac{\mu(B_r)}{r^{sp}}.
\]
 and 
\[
\int_{B_{r}\cup [B_{2r}\setminus B_{r+\varepsilon r}]}\left(\int_X
\frac{|u(x)-u(y)|^{q}}
     {d(x,y)^{s q}\mu(B(x,d(x,y)))}\,d\mu(y)\right)^\frac{p}{q}d\mu(x)\lesssim\frac{\mu(B_r)}{(\varepsilon r)^{sp}}.
\]

Let $sp<\eta$. For $x\in B_{r}$, fix $r_x= \dist(x,\partial B_r)$ and set $A_j(x)=B(x,2^{j+1}r_x)\setminus B(x,2^jr_x)$, $j\in\N\cup\{0\}$.
Then, we have
\begin{align}
    \int_{B_{r}}&\left(\int_{ B_r^c } \frac{|u(x)-u(y)|}{d(x,y)^{sq}\mu(B(x,d(x,y)))}    \,d\mu(y)\right)^\frac{p}{q}\,d\mu(x)\\
    &\leq \int_{B_{r}}\left(\int_{B(x,\dist(x,\partial B_r))^c} \frac{|u(x)-u(y)|}{d(x,y)^{sq}\mu(B(x,d(x,y)))}    \,d\mu(y)\right)^\frac{p}{q}\,d\mu(x)\\
  & \le \int_{B_{r}}\left(\sum_{j=0}^N\int_{A_j(x)} \frac{1}{(2^jr_x)^{sq}\mu(B(x,2^jr_x))}    \,d\mu(y)\right)^\frac{p}{q}\,d\mu(x)\\
  & \lesssim \int_{B_{r}} \dist(x,\partial B_r)^{-sp}d\mu(x)\\
  & \le \sum_{j=1}^\infty \frac{2^{jsp}}{ r^{sp}}\mu(\{x\in B_{r} : 2^{-j} r\le \dist(x,\partial B_r)< 2^{-j+1} r\}) \\
   & \lesssim\sum_{j=1}^\infty \frac{2^{jsp}}{r^{sp}}\mu(B(O,r-2^{-j} r)\setminus B(O,r-2^{-j+1} r) )\lesssim\frac{\mu(B_r)}{r^{sp}}.
\end{align} 
By a similar argument, we have
\begin{equation}\begin{split}
     \int_{B_{2r}\setminus B_{r+\varepsilon r}}&\left(\int_{B_{r+\varepsilon r}
     } \frac{|u(x)-u(y)|}{d(x,y)^{sq}\mu(B(x,d(x,y)))}    \,d\mu(y)\right)^\frac{p}{q}\,d\mu(x)\\
     &\lesssim \sum_{j=1}^\infty \frac{2^{jsp}}{r^{sp}}\mu(B(O,r+2^{-j} r)\setminus B(O,r+2^{-j+1} r) )\lesssim\frac{\mu(B_r)}{r^{sp}}.
\end{split}
\end{equation}
Observe that the above two estimates blow up when $sp\geq\eta$. Therefore, for $sp\geq\eta$ we have 
\begin{equation}
    \mathrm{cap}_{s,p,q}(B_r,B_{r+\varepsilon r})
\lesssim
  \varepsilon^{\eta-sp}\frac{\mu(B_r)}{r^{sp}}+\frac{\mu(B_r)}{r^{sp}}\left(1+\frac{1}{\varepsilon^{sp}}\right),
\end{equation}
    and for $sp<\eta$ we have
    \begin{align}
         \mathrm{cap}_{s,p,q}(B_r,B_{r+\varepsilon r})
&\lesssim
  \varepsilon^{\eta-sp}\frac{\mu(B_r)}{r^{sp}}+\frac{\mu(B_r)}{r^{sp}}\min\left\{1,\frac{1}{\varepsilon^{sp}}\right\}\\
 & \lesssim
  \varepsilon^{\eta-sp}\frac{\mu(B_r)}{r^{sp}}+\frac{\mu(B_r)}{r^{sp}}.
    \end{align}
    Hence the claim follows.
\end{proof}

The following theorem provides two-sided bounds for the capacity of thin annuli.

\begin{theorem}\label{thin annuli_main}
     Let $0<s<1$, $0<q\le p<\infty$, and assume that the measure $\mu$ satisfies the $\eta$-annular decay property.  Let $0<r<\frac{1}{4}\diam(X)$, $0<\varepsilon<1$, $\lambda>1$ with $\lambda \varepsilon<1$  and $E\subset B_{r+\lambda\varepsilon r}\setminus B_{r+\varepsilon r}$ is such that $\mu(E)\geq c\mu(B_{r+\lambda \varepsilon r}\setminus B_{r+\varepsilon r})$ for some $c\in (0,1]$.  In addition, assume that there exists $\lambda'>1$ such that for every $x\in E $ there exists a point $x'$ such that $B(x',\varepsilon r/2)\subset B(x,\lambda'\varepsilon r)\cap B_r$.
 Then, for $sp\ge \eta$ we have
   \begin{equation}\label{eq0_5.10}
(\varepsilon^{\eta-sp}+1)\frac{\mu(B_r)}{r^{sp}}\lesssim\mathrm{cap}_{s,p,q}(B_r,B_{r+\varepsilon r})
\lesssim
( \varepsilon^{\eta-sp}+\varepsilon^{-sp}+1)\frac{\mu(B_r)}{r^{sp}},
\end{equation}
and  for $sp< \eta$ we have
\begin{equation}
    \mathrm{cap}_{s,p,q}(B_r,B_{r+\varepsilon r})\simeq \frac{\mu(B_r)}{r^{sp}},
\end{equation}
where the comparison constant depends on $s$, $p$, $\eta$, $c_A$, $\lambda'$, $c$, and $C(c_D,c_{RD},C_{DT})^{1/q}$.  

Moreover, the above estimates also hold for $0<p<q\le\infty$, provided that the measure $\mu$ satisfies the dyadic tail condition \eqref{eq2_6.9}.
\end{theorem}
\begin{proof}The proof follows from Lemma~\ref{thin annuli_ub} and Lemma~\ref{thin annuli_lb} together with $\eta$-annular decay property.
\end{proof}

    \begin{remark}
If $sp<\eta$, the capacity of thin annuli is positive and finite. If $sp\ge\eta$, we obtain
sharp capacity bounds whenever $\varepsilon$ is bounded away from $0$, and the
capacity of thin annuli blows up as $ \varepsilon\to0$.
\end{remark}

\begin{corollary}
Under the assumptions of Theorem~\ref{thin annuli_main}, if $sp=s'p'$ and $0<q,q'\le\infty$, then $\mathrm{cap}_{s,p,q}(B_r,B_{r+\varepsilon r})\simeq\mathrm{cap}_{s',p',q'}(B_r,B_{r+\varepsilon r})\simeq\mathrm{cap}_{s,p,p}(B_r,B_{r+\varepsilon r})$.
\end{corollary}

\subsection{Capacity characterization using indicator functions}
In this section, we characterize the capacity of fat and thin annuli using
indicator functions. We first establish Lemma~\ref{lem0_6.10}, which gives
conditions under which indicator functions of balls belong to the fractional
Sobolev space $W^s_{p,q}(X)$. Using it, we prove
Corollary~\ref{coro_14.9} in a more general setting.

\begin{lemma}\label{lem0_6.10}
   Let $0<s<1$ and $0<q\le p<\infty$. Assume that the measure $\mu$ satisfies the upper $\eta$-annular decay property and $sp<\eta$.  Then for $r>0$ we have
   \begin{equation}
  [\1_{B_r}]_{W^s_{p,q}(X)}^p\lesssim\frac{\mu(B_r)}{r^{sp}},
   \end{equation}
   where the comparison constant depends on $s$, $p$, $\eta$, $c_A$, and $C(c_D,c_{RD},C_{DT})^{1/q}$.
   
   Moreover, the above estimates also hold for $0<p<q\le\infty$, provided that the measure $\mu$ satisfies the dyadic tail condition  \eqref{eq2_6.9}.
\end{lemma}

\begin{proof}
    Let $sp<\eta$ and $r>0$. Take $u=\1_{B_r}$. Then 
    \begin{align}
    [u]_{W^s_{p,q}(X)}^p&\lesssim \int_{B_{r}}\left(\int_{B_r^c} \frac{|u(x)-u(y)|}{d(x,y)^{sq}\mu(B(x,d(x,y)))}    \,d\mu(y)\right)^\frac{p}{q}\,d\mu(x)\\
     &\qquad+\int_{B_{2r}\setminus B_r}\left(\int_{B_r} \frac{|u(x)-u(y)|}{d(x,y)^{sq}\mu(B(x,d(x,y)))}    \,d\mu(y)\right)^\frac{p}{q}\,d\mu(x)\\
    &\quad\qquad+\int_{B_{2r}^c}\left(\int_{B_r} \frac{|u(x)-u(y)|}{d(x,y)^{sq}\mu(B(x,d(x,y)))}    \,d\mu(y)\right)^\frac{p}{q}\,d\mu(x)\\
   &=:I_1+I_2+I_3.
   \end{align}
   For $x\in B_r$, fix $r_x= \dist(x,\partial B_r)$ and set $A_j(x)=B(x,2^{j+1}r_x)\setminus B(x,2^jr_x)$, $j\in\N\cup\{0\}$. Then for $I_1$ we have
\begin{align}\label{I1_14.9}
    I_1 &\leq \int_{B_{r}}\left(\int_{B(x,\dist(x,\partial B_r))^c} \frac{|u(x)-u(y)|}{d(x,y)^{sq}\mu(B(x,d(x,y)))}    \,d\mu(y)\right)^\frac{p}{q}\,d\mu(x)\\
  & \le \int_{B_{r}}\left(\sum_{j=0}^N\int_{A_j(x)} \frac{1}{(2^jr_x)^{sq}\mu(B(x,2^jr_x))}    \,d\mu(y)\right)^\frac{p}{q}\,d\mu(x)\\
  & \lesssim \int_{B_{r}} \dist(x,\partial B_r)^{-sp}d\mu(x)\\
  & \le \sum_{j=1}^\infty \frac{2^{jsp}}{r^{sp}}\mu(\{x\in B_r : 2^{-j}r\le \dist(x,\partial B_r)< 2^{-j+1}r\}) \\
   & = \sum_{j=1}^\infty \frac{2^{jsp}}{r^{sp}}\mu(B(O,r-2^{-j}r)\setminus B(O,r-2^{-j+1}r) )
   \\
  &\lesssim  \frac{\mu(B(O,r))}{r^{sp}} \sum_{j=1}^\infty \frac{2^{jsp}}{2^{j\eta}}\lesssim\frac{\mu(B(O,r))}{r^{sp}},
\end{align}
 since $sp<\eta$. Using a similar argument, we obtain $I_2\lesssim\frac{\mu(B(O,r))}{r^{sp}}$.    Set $A_j=B(O,2^{j+1}r)\setminus B(O,2^jr)$, $j\in\N$. Then, for $I_3$ we have
     \begin{align}
I_3 &\lesssim  \sum_{j=1}^N \int_{A_j}\left(\int_{B_{r}}
 \frac{1}
      {d(x,y)^{s q}\mu(B(x,d(x,y)))}\,d\mu(y)\right)^\frac{p}{q}d\mu(x)\\
      &\lesssim \sum_{j=1}^N \int_{A_j}\left(
\frac{\mu(B_{r})}{(2^jr)^{sq}\mu(B(x,2^jr))}\right)^\frac{p}{q}d\mu(x)\\
&\lesssim \sum_{j=1}^N \frac{\mu(B_{r})^\frac{p}{q}\mu(B(O,2^{j}r))^{1-\frac{p}{q}}}{(2^jr)^{sp}}\\
 &\lesssim \frac{\mu(B_{r})}{r^{sp}},
\end{align}
 where the last inequality follows from \eqref{eq2_6.9} when $p<q$ and from \eqref{iteration reverse} when $p\ge q$. Hence, the claim follows.
\end{proof}
Let us single out the following consequences.

 \begin{corollary}[Theorem~\ref{sharp capacity estimate annuli}]\label{ind_fat annuli}
        Let $0<s<1$ and $0<q\le p<\infty$. Assume that the measure $\mu$ satisfies the upper $\eta$-annular decay property and $sp<\eta$.
 Then, for $0<2r\le R<\frac{1}{4}\diam(X)$ we have
   \begin{equation}
              \mathrm{cap}_{s,p,q}(B_r,B_R)\simeq 
              \begin{cases}
              \displaystyle
     \left(\int_{r}^{R}\,[\1_{B_{\hat{r}}}]_{{W}^s_{p,q}(X)}^{\frac{- p}{p-1}}\frac{d\hat{r}}{\hat{r}}\right)^{1-p }, & p>1,\\ \\
\displaystyle
   \underset{r\leq \hat r<2\hat r\leq R}{\inf}\,[\1_{B_{\hat{r}}}]_{{W}^s_{p,q}(X)}^p\qquad, & p\le 1,
\end{cases}
            \end{equation}
where the comparison constant depends on $s$, $p$, $\eta$, $c_A$, and $C(c_D,c_{RD},C_{DT})^{1/q}$.

Moreover, the above estimates also hold for $0<p<q\le\infty$, provided that the measure $\mu$ satisfies the dyadic tail condition \eqref{eq2_6.9}.
 \end{corollary}
\begin{proof}
     Let   $0<r\le \hat{r}< 2\hat{r}\le R< \frac{1}{4}\diam(X)$. Since $u=\1_{B_{\hat{r}}}$ is admissible for $\mathrm{cap}_{s,p,q}(B_{\hat{r}},B_{2\hat{r}})$, it follows from Lemma~\ref{lem0_6.10}, Corollary~\ref{lem0_4.9.26} and Lemma~\ref{lem0_6.9},
 \begin{equation}
 \frac{\mu(B_{\hat{r}})}{\hat{r}^{sp}} \simeq  \mathrm{cap}_{s,p,q}(B_{\hat{r}},B_{2\hat{r}})\simeq     [\1_{B_{\hat{r}}}]_{W^s_{p,q}(X)}^p.
 \end{equation}
   Hence, the claim follows by combining the above estimate with Theorem~\ref{sharp capacity estimate annuli}. 
\end{proof}

 \begin{corollary}[Theorem~\ref{thin annuli_main}] \label{ind_thin annuli}
Under the assumptions of Theorem~\ref{thin annuli_main}, we have
   \begin{equation}
              \mathrm{cap}_{s,p,q}(B_r,B_{r+\varepsilon r})\simeq [\1_{B_{r}}]_{{W}^s_{p,q}(X)}^p,
            \end{equation}
           where the comparison constant depends on $s$, $p$, $\eta$, $c_A$, and $C(c_D,c_{RD},C_{DT})^{1/q}$.
\end{corollary} 
 
 \begin{proof}
     The proof is similar to that of Corollary~\ref{ind_fat annuli}.
 \end{proof}

\begin{proof}[Proof of Corollary~\ref{coro_14.9}]
    The proof follows from Corollary~\ref{ind_fat annuli} and Corollary~\ref{ind_thin annuli}.
\end{proof}

\section*{Acknowledgements}

A. A. was supported by the Academy of Finland via the Centre of Excellence in Analysis and Dynamics Research (Project number 323960). The author would like to thank her advisor, Pekka Koskela, for fruitful discussions and valuable suggestions. She also thanks Khanh Nguyen for his helpful remarks on the paper.
\newpage

{\bf Conflict of interest statement}: None declared.

{\bf Data availability}: 
No data were used for the research described in the article.

{\bf AI declaration}: 
AI has not been used in this paper.

 \bibliographystyle{plain}
			\bibliography{bibliography.bib}
			
			\setlength{\parindent}{0pt}

\end{document}